\documentclass[10pt,twocolumn,twoside]{IEEEtran}
\usepackage[letterpaper, portrait, margin=1in]{geometry}
\usepackage[utf8]{inputenc}
\usepackage{amsmath}
\usepackage{amssymb}
\usepackage{amsthm}
\usepackage{graphicx}
\usepackage{hyperref}
\usepackage{xcolor}

\DeclareMathOperator*{\argmax }{argmax}
\DeclareMathOperator*{\argmin }{argmin}
\DeclareMathOperator*{\argstat }{argstat}

\usepackage{tikz}
\usetikzlibrary{shapes.geometric,arrows,chains,scopes}

\tikzstyle{startstop} = [rectangle, rounded corners, minimum width=3cm, minimum height=1cm,text centered, draw=black, fill=red!30]
\tikzstyle{io} = [trapezium, trapezium left angle=70, trapezium right angle=110, minimum width=2.5cm, minimum height=1cm, text centered, draw=black, fill=blue!30]
\tikzstyle{process} = [rectangle, minimum width=3cm, minimum height=1cm, text centered, draw=black, fill=orange!30]
\tikzstyle{decision} = [diamond, aspect=2, minimum width=3cm, minimum height=1cm, text centered, text width=2.5cm, draw=black, fill=green!30]
\tikzstyle{arrow} = [thick,->,>=stealth]
\tikzstyle{coord} = [coordinate, on chain, on grid, node distance=6mm and 25mm]

\newcommand{\range}{\mathrm{range}}

\theoremstyle{definition}
\newtheorem{definition}{Definition}[section]

\newtheorem{problem}{Problem}
\newtheorem{theorem}{Theorem}[section]

\newtheorem{lemma}[theorem]{Lemma}

\title{A unified analytical framework for a class of optimal control problems on networked systems}
\author{Mingwu Li, Harry Dankowicz
\thanks{Manuscript received $\cdots$. \emph{Corresponding author: Mingwu Li} }
\thanks{The authors are with the Department of Mechanical Science and Engineering, University of Illinois at Urbana-Champaign, Urbana, IL, 61801, USA (e-mail: mingwul2@illinois.edu, danko@illinois.edu).}}

\date{}

\begin{document}

\maketitle

\begin{abstract}
We consider a class of optimal control problems on networks that generically permits a reduction to a universal set of reference problems without differential constraints that may be solved analytically. The derivation shows that input homogeneity across the network results in universally constant optimal control inputs. These predictions are validated using numerical analysis of problems of synchronization of coupled phase oscillators and spreading dynamics on time-varying networks.
\end{abstract}


\begin{IEEEkeywords}
Nonlinear Systems, optimal Control, networks of autonomous agents, stochastic/uncertain systems, continuation Methods
\end{IEEEkeywords}

\section{Introduction}

In network science, one is often interested not only in the (possibly time-dependent) topological properties of the network, but also in the dynamics of processes occurring on the network. In particular, one may seek to design \emph{optimal control inputs} to manipulate these dynamical processes to achieve a desired outcome within certain constraints. The purpose of this paper is to investigate such optimal control designs in two commonly studied processes, viz., \emph{synchronization} (or consensus) and \emph{spreading} under assumptions of input homogeneity across the network.

Synchronization phenomena are observed widely in nature and science~\cite{synchronization-auto,2016kuramoto}. The \emph{Kuramoto} model, first proposed in 1975~\cite{kuramoto1975,synchronizationBook}, offers a theoretical template for analyzing topology-dependent synchronization in complex networks. In this model, a network of coupled oscillators exhibits relative phase dynamics governed by a distribution of natural frequencies and perturbed by sinusoidal functions of phase differences between neighboring nodes. \textcolor{black}{Notable studies of synchronization in versions of this model with constant coupling strengths include  refs.~\cite{strogatzLinearIncoherent,crawford1994,ott}. In~\cite{leander2015controlling}, time-dependent, heterogeneous coupling strengths derived using optimal control theory are shown to produce desired synchrony more efficiently than is possible in the time-independent case. To avoid the curse of dimensionality that plagues~\cite{leander2015controlling} (the number of control inputs is a quadratic function of the size of the network), we focus on the case of homogeneous coupling and explore whether the advantage of dynamical coupling still holds.}


Spreading is a process that occurs in various guises in real-world networks, from disease propagation in animal social networks to information dissemination in political and marketing campaigns. Over the past decade, many efforts have sought to understand the influence of \emph{temporal patterns} on spreading; in empirical settings~\cite{slowdown,lifetime_reference}, in synthetic networks \cite{lifetime_reference,vertexburst} and in theoretical models~\cite{non-station}. 
The optimization and control of epidemics on static networks has been well explored~\cite{epidemic_control,nowzari2017optimal,yang2016optimal}, while optimization studies in the setting of time-varying networks are limited.

In this paper, we consider a special case of optimal control design on networks governed by dynamical systems of the form $\dot{y}=u(t)f(y)$ in terms of the scalar-valued, time-dependent control input $u(t)$. We obtain examples from the optimal control of synchronization of identical, coupled phase oscillators on static networks with time-varying coupling strength and spreading dynamics on time-varying activity-driven networks. The latter have recently emerged as a powerful paradigm to study epidemic spreading over realistic networks~\cite{activity-driven,zino2016}. Although the contexts of these optimization problems are very different, the universal form of the governing equations allows their analysis using a common analytical framework developed in this study.

We are generally not able to solve optimal control problems explicitly since analytical solutions are rarely available for constraints given by differential equations. Interestingly, by assuming the homogeneous application of the control input across all system states, we show that the search for optimal control inputs may be equivalently (under suitable non-degeneracy conditions) considered in the context of reference optimization problems with only integral constraints on the control input. The relationships between the original problem formulations and the reference optimization problems is here expressed in terms of a time-rescaled dynamical system $y'=f(y)$ with control-independent solutions.

Mirroring the formulation of three distinct optimal control problems, viz., ``maximum synchronization/spreading'', ``minimum effort'', and ``minimum time'', we arrive at three equivalent reference problems, each of which may be solved directly by analysis of the integrand of the cost functional. In particular, this analysis shows that when equivalence applies between the original problems and the reference optimization problems, candidate optimal solutions must be constant functions of time. Thus, under the homogeneity assumption, the flexibility of time dependence does not yield improved performance. 

The remainder of this paper is organized as follows. In Section~\ref{sec:opt-formulation}, we consider the optimal control of finite networks of coupled phase oscillators in the context of the three optimization problems mentioned above. We apply a continuation-based optimization technique to obtain candidate optimal control inputs that are found in all cases to be constant functions of time. Motivated by this observation, an analytical framework for a broader class of network optimal control problems is developed in Section~\ref{sec:analytical}, verifying the observations from Section~\ref{sec:opt-formulation}. We first describe the equivalent reference optimization problems and proceed to establish the equivalence between the corresponding necessary conditions for stationary points under a non-degeneracy condition on a suitably formulated Lie derivative at the terminal time of integration. For a particular choice of cost functional, we use the reference problems to find explicit solutions to the original optimization problems and apply numerical continuation techniques to validate these predictions on the dependence of problem parameters. Section~\ref{sec:two-more-appl} presents the application of the framework to two other dynamical processes, namely, synchronization of phase oscillators in the continuum limit, and spreading dynamics on activity driven networks. Section~\ref{sec:conclusion} concludes this paper with a brief summary and several directions for future study.

\section{Optimization of synchronization of coupled phase oscillators}
\label{sec:opt-formulation}
\subsection{System dynamics}
Consider a Kuramoto model of $N$ coupled phase oscillators with identical natural frequencies $\omega$~\cite{HongStrogatz} on an undirected and unweighted network with adjacency matrix $A$, such that
\begin{equation}
\label{eq:ode-degree}
\dot{x}_i=\mu(t)\sum_{j=1}^N a_{ij}\sin(x_j-x_i),\quad i=1,\cdots,N
\end{equation}
where $x_i+\omega t$ is the phase of the $i$-th oscillator, $a_{ij}\in\{0,1\}$ denotes the $(i,j)$-th entry of $A$, and $\mu(t)\in\mathbb{R}_+$ is a \emph{time-varying} coupling strength. 

\subsection{Characterization of synchronization}
A natural candidate for characterizing the level of synchronization is the amplitude $|\hat{r}|$ of the centroid
\begin{equation}
\hat{r}: =\frac{1}{N}\sum_{j=1}^N e^{\mathrm{i}x_j}
\end{equation}
in the complex plane of the group of phase oscillators. We speak of \emph{incoherent collective motion} when $|\hat{r}|=0$, \emph{partial synchronization} when $|\hat{r}|\in(0,1)$, and \emph{complete synchronization} when $|\hat{r}|=1$.

\subsection{Problem formulation}
\label{sec:optProbForm}
With $\mu(t)$ as the control input and for given initial conditions, we consider three distinct optimal control problems in terms of the cost functional
\begin{equation}
    \mathcal{G}(\mu):=\int_0^Tg(\mu(t))\,\mathrm{d}t
\end{equation}
for some scalar-valued, positive function $g$.

\begin{problem}[Maximum synchronization]\label{Maximizing consensus-Dynamic}
Given $T,C_1\in\mathbb{R}_+$, find
\begin{align}
\label{eq:p1}
 & \argmax_{\mu\in C([0,T],\mathbb{R}_+)} |\hat{r}(T)| \quad \text{s.t.}\quad \mathcal{G}(\mu)= C_1.
\end{align}
\end{problem}

\begin{problem}[Minimum effort]\label{Minimize resource allocation-Dynamic}
Given $T,\underline{r}\in\mathbb{R}_+$, find
\begin{align}
\label{eq:p2}
 & \argmin_{\mu\in C([0,T],\mathbb{R}_+)} \mathcal{G}(\mu) \quad\text{s.t.}\quad|\hat{r}(T)|=\underline{r}.
\end{align}
\end{problem}

\begin{problem}[Minimum time]\label{Minimum time control-Dynamic}
Given $C_1,\underline{r}\in\mathbb{R}_+$, find
\begin{align}
\label{eq:p3}
 & \argmin_{\mu\in C([0,T],\mathbb{R}_+)} T \quad\text{s.t.}\quad \mathcal{G}(\mu)=C_1, \,|\hat{r}(T)|=\underline{r}.
\end{align}
\end{problem}


\subsection{Preliminary numerical results}
\label{sec:sim-finite-motivation}
\subsubsection{Optimization algorithm}
We use a continuation-based framework developed in~\cite{DKK91,staged_adjoint} and implemented in the software package \textsc{coco}~\cite{coco,coco-recipes} to find candidate solutions to these optimal control problems. A brief introduction to this framework is given as follows.
\begin{itemize}
\item \textbf{Step 1}: \emph{Control parameterization}. We restrict attention to truncated polynomial expansions, $\mu(\sigma):=\sum_{i=1}^q p_i\hat{T}_{i}(\sigma)$, in terms of a set of \emph{normalized Chebyshev polynomials of the first kind} with $\hat{T}_1(\sigma)=1/\sqrt{\pi}$ and $\sigma:=2t/T-1$ for $t\in[0,T]$. The coefficients $\{p_i\}_{i=1}^q$ become design variables.
\item \textbf{Step 2}: \emph{Necessary conditions}. The necessary conditions for constrained optimization couple the original integral and differential constraints to a set of \emph{adjoint equations} in terms of unknown \emph{Lagrange multipliers}.  \textsc{coco} provides a predefined library of adjoint constructors for common types of integral, differential, and algebraic constraint operators. This library provides building blocks for staged construction of the complete set of adjoint equations for constrained optimization of integro-differential boundary-value problems~\cite{staged_adjoint}.
\item \textbf{Step 3}: \emph{Successive continuation}. Without restriction to positive-valued $\mu$, we use parameter continuation to satisfy the necessary conditions through a succession of separate stages~\cite{DKK91,staged_adjoint,inequality-ncp}, where each successive run is initialized by the solution from the previous run. Importantly, due to the linearity and homogeneity of the adjoint equations, the first run can be initialized with a solution guess with \emph{zero} Lagrange multipliers. 
\end{itemize}

\subsubsection{Example} Let $N=10$ and consider the example graph in Fig.~\ref{fig:graph-plot}. Suppose that $x_i(0)=2\pi(i-1)/10$ for $1\leq i\leq 10$ such that $\hat{r}(0)=0$. Let $g(\mu)=\mu^2$ and consider polynomial expansions with $q=10$ and initial solution guess $p_1=1$, $p_{2\ldots 10}=0$.

\begin{figure}[ht]
\centering
\includegraphics[width=2.3in]{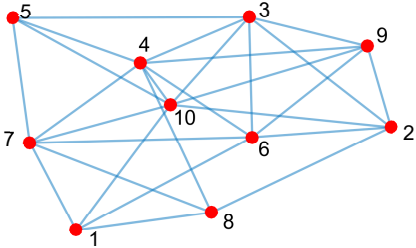}
\caption{An example network graph on 10 nodes.}
\label{fig:graph-plot}
\end{figure}

Let $T=3$ and $C_1=1$ in Problem~\ref{Maximizing consensus-Dynamic}. Using the optimization algorithm, we locate the five stationary points (SPs) listed in Table~\ref{FPs-P1} in the computational domain defined by $0.01\leq|\hat{r}(T)|\leq1$. The first and fourth of these points share the same value for the objective function $|\hat{r}(T)|$ but correspond to different $\mu$. This holds also for the second and the fifth SPs. In all four instances, $\mu$ is time-dependent and negative on a subset of $[0,T]$. In contrast, $\mu$ is positive and constant for the third SP, which also corresponds to the maximum value of $|\hat{r}(T)|$.

\begin{table}[h]
\begin{center}
\begin{tabular}{ c|c|c|c|c } 
 \hline
  SP & $|\hat{r}(T)|$ & $p_1$ & $p_2$ & $p_{3,\cdots,10}$\\ 
 \hline
 1 & 0.0177 & 0.3576 & 1.1743 & 0\\ 
 2 & 0.0327 & 0.2232 & 1.2231 & 0\\ 
 3 & 0.8889 & 1.0233 & 0 & 0\\ 4 & 0.0177 & 0.3576 & -1.1743 & 0\\ 
 5 & 0.0327 & 0.2232 & -1.2231 & 0\\ 
 \hline
\end{tabular}
\caption{Five stationary points (SPs) found for Problem~\ref{Maximizing consensus-Dynamic}.}
\label{FPs-P1}
\end{center}
\end{table}

Let $T=3$ and $\underline{r}=0.9$ in Problem~\ref{Minimize resource allocation-Dynamic}. In this case, we locate a single stationary point with  $p_1\approx 1.0356$ and $p_{2,...,10}\approx 0$. The corresponding $\mu$ is again constant.

Finally, let $C_1=1$ and $\underline{r}=0.9$ in Problem~\ref{Minimum time control-Dynamic}. We again locate only one stationary point. The minimum value $T=3.0721$ is obtained when $p_1\approx 1.0112$ and $p_{2,...,10}\approx 0$. The corresponding $\mu$ is again constant.

The search for optimal solutions does not appear to benefit from the flexibility of time-dependence of the control input. The rigorous analysis in the next section explains this observation and shows how the results may be generalized to other network optimization problems.

\section{Analytical framework}
\label{sec:analytical}
\subsection{Three reference optimization problems}
The three optimal control problems in Section~\ref{sec:optProbForm} are closely related to the following three reference problems in terms of the functional
\begin{equation}
    \mathcal{I}(\mu):=\int_0^T\mu(t)\,\mathrm{d}t.
\end{equation}

\begin{problem}\label{P1}
Given $T,C_1\in\mathbb{R}_+$, find
\begin{equation}
    \argstat_{\mu\in C([0,T],\mathbb{R}_+)}\mathcal{I}(\mu)
    \quad \text{s.t.}\quad \mathcal{G}(\mu)= C_1.
\end{equation}
\end{problem}

\emph{Solution:} In terms of the Lagrangian
\begin{equation}
L := \int_0^T\mu(t)\,\mathrm{d}t + \lambda_{\mathrm{ref}}\left(\int_0^Tg(\mu(t))\,\mathrm{d}t-C_1\right)
\end{equation}
and the Lagrange multiplier $\lambda_\mathrm{ref}\in\mathbb{R}$, $\delta L(\mu(t),\lambda_\mathrm{ref})=0$ implies that
\begin{equation}
\label{eq:adj-p1}
\mathcal{G}(\mu)=C_1,\, 1+\lambda_{\mathrm{ref}}g'(\mu(t)) = 0.
\end{equation}
Lemma~\ref{sec:sol2p1} implies that a locally unique solution $\mu$ is a positive constant $\mu^\ast$ under some non-degeneracy conditions on $g$ at $\mu^\ast$. The integral constraint reduces to $Tg(\mu^\ast)=C_1$, which may be inverted to find $\mu^\ast$.\qed

With $g(\mu)=\mu^2$ we obtain
\begin{equation}
\label{eq:sol2p4}
\mu^\ast = \sqrt{C_1/T},\,\lambda_{\mathrm{ref}} = -\sqrt{T/C_1}/2.
\end{equation}

\begin{problem}\label{P2}
Given $T,C_2\in\mathbb{R}_+$, find
\begin{equation}
    \argstat_{\mu\in C([0,T],\mathbb{R}_+)}\mathcal{G}(\mu)
    \quad \text{s.t.}\quad \mathcal{I}(\mu)= C_2.
\end{equation}
\end{problem}

\emph{Solution:} In terms of the Lagrangian
\begin{equation}
L := \int_0^Tg(\mu(t))\,\mathrm{d}t + \lambda_{\mathrm{ref}}\left(\int_0^T\mu(t)\,\mathrm{d}t-C_2\right)
\end{equation}
and the Lagrange multiplier $\lambda_\mathrm{ref}\in\mathbb{R}$,
$\delta L(\mu(t),\lambda_\mathrm{ref})=0$ implies that
\begin{equation}
\label{eq:adj-p2}
\mathcal{I}(\mu)=C_2,\, \lambda_{\mathrm{ref}}+g'(\mu(t)) = 0.
\end{equation}
Lemma~\ref{sec:sol2p2} implies that a globally unique solution $\mu$ is a positive constant $\mu^\ast$ under some non-degeneracy conditions on $g$ at $\mu^\ast$. The integral constraint reduces to $T\mu^\ast=C_2$, which may be inverted to find $\mu^\ast$, \emph{independently} of $g$.\qed

With $g(\mu)=\mu^2$ we obtain
\begin{equation}
\label{eq:sol2p5}
\mu^\ast = C_2/T,\,\lambda_{\mathrm{ref}} = -2C_2/T.
\end{equation}

\begin{problem}\label{P3}
Given $C_1,C_2\in\mathbb{R}_+$, find
\begin{equation}
    \argstat_{\mu\in C([0,T],\mathbb{R}_+)}T
    \quad \text{s.t.}\quad \mathcal{G}(\mu)=C_1,\,\mathcal{I}(\mu)= C_2.
\end{equation}
\end{problem}

\emph{Solution:} In terms of the Lagrangian
\begin{align}
L & := T+\lambda_{1,\mathrm{ref}}\left(\int_0^Tg(\mu(t))\,\mathrm{d}t-C_1\right) \nonumber\\
 & \quad+ \lambda_{2,\mathrm{ref}}\left(\int_0^T\mu(t)\,\mathrm{d}t-C_2\right)
\end{align}
and the Lagrange multipliers $\lambda_{1,\mathrm{ref}}\in\mathbb{R}$ and $\lambda_{2,\mathrm{ref}}\in\mathbb{R}$,
$\delta L(\mu(t),T,\lambda_{1,\mathrm{ref}},\lambda_{2,\mathrm{ref}})=0$ implies that
\begin{gather}
\label{eq:adj-p3-l1}
\mathcal{G}(\mu)=C_1,\, \mathcal{I}(\mu)=C_2,\\ \lambda_{1,\mathrm{ref}}g'(\mu(t))+\lambda_{2,\mathrm{ref}}=0,\\
\label{eq:adj-p3-T}
\quad 1+\lambda_{1,\mathrm{ref}}g(\mu(T))+\lambda_{2,\mathrm{ref}}\mu(T)=0.
\end{gather}
Lemma~\ref{sec:sol2p3} implies that a locally unique solution $\mu$ is a positive constant $\mu^\ast$ under some non-degeneracy conditions on $g$ at $\mu^\ast$. The two integral equations then reduce to two algebraic conditions, namely, $Tg(\mu^\ast)=C_1$ and $T\mu^\ast=C_2$, which may be inverted to find $\mu^\ast$ and $T$.\qed

With $g(\mu)=\mu^2$ we obtain
\begin{gather}
\mu^\ast=C_1/C_2,\, T=C_2^2/C_1,\label{eq:sol2p6}\\ \lambda_{1,\mathrm{ref}} = C_2^2/C_1^2,\, \lambda_{2,\mathrm{ref}} = -2C_2/C_1.\label{eq:sol2p6-adj}
\end{gather}

\subsection{Three essential theorems}
\label{par:reformulation}
The three theorems in this section establish the relationship between the three optimal control problems in Section~\ref{sec:optProbForm} and the three reference problems.

\begin{definition}
Consider a non-autonomous system on $\mathbb{R}^n$ of the form $\dot{z}=\mu(t)h(z,p)$, where  $\mu:[0,\infty)\to\mathbb{R}_+$, $h:\mathbb{R}^n\times\mathbb{R}^s\mapsto\mathbb{R}^n$, and $p$ are the system parameters. We say that such a dynamical system is \emph{separable}.
\end{definition}

For a separable dynamical system, the following result follows from the fact that $\mu(t)$ is positive. 
\begin{lemma}
\label{lema:odes}
The invertible time-rescaling $\tau(t):=\int_0^t \mu(s)\,\mathrm{d}s$ transforms the separable non-autonomous system $\dot{z}=\mu(t)h(z,p)$ to the autonomous system 
\begin{equation}
\label{eq:ref-auto}
\hat{z}'=h(\hat{z},p),
\end{equation}
where $\hat{z}(\tau)=z(t(\tau))$.
\end{lemma}

The Kuramoto model in \eqref{eq:ode-degree} is separable with $z=(x_1,\cdots,x_N)$ and $p=\emptyset$. We may represent its dynamics in the rescaled time domain, as exemplified in Fig.~\ref{fig:Finite_rescaled_centroid-finite}.

\begin{figure}[ht]
\centering
\includegraphics[width=2.8in]{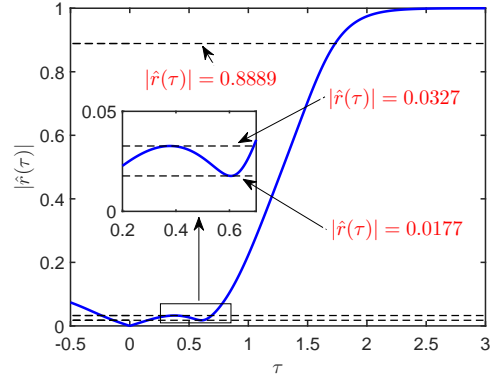}
\caption{Time-rescaled dynamics of the Kuramoto model in \eqref{eq:ode-degree} on the graph in Fig.~\ref{fig:graph-plot} with $x_i(0)=2\pi(i-1)/10$ for $i=1,\ldots,10$.}
\label{fig:Finite_rescaled_centroid-finite}
\end{figure}

Suppose that 
\begin{equation}
\label{eq:sepOdeZ}
\dot{z}=\mu(t)h(z,p),\quad z(0)=z_0, \quad t\in[0,T]
\end{equation}
and consider a function $\Phi:\mathbb{R}^n\rightarrow\mathbb{R}$.

\begin{theorem}
\label{theo1}
Given $T,C_1\in\mathbb{R}_+$, suppose that $\mu^\ast\in C([0,T],\mathbb{R}_+)$ satisfies the first-order necessary conditions for a stationary point of the functional $\mu\mapsto\Phi(z(T))$ subject to $\mathcal{G}(\mu)=C_1$. Then, if
\begin{equation}
    \Phi_h:= \nabla\Phi(z(T))\cdot h(z(T),p)\neq0,
\end{equation}
it follows that $\mu^\ast$ satisfies the first-order necessary conditions of Problem~\ref{P1} and vice versa.
\end{theorem}
\begin{proof}
By definition,
\begin{equation}
    z(T)=\hat{z}\left(\mathcal{I}(\mu)\right) = z_0+\int_0^{\mathcal{I}(\mu)} h(\hat{z}(\tau),p)\,\mathrm{d}\tau
\end{equation}
and, consequently,
\begin{align}
\label{eq:var-zhat}
\delta\hat{z}\left(\mathcal{I}(\mu)\right) &=h(\hat{z}\left(\mathcal{I}(\mu)\right),p) \int_0^T \delta \mu(t)\,\mathrm{d}t.
\end{align}
In terms of the Lagrangian
\begin{equation}
L := \Phi(z(T)) + \lambda\left(\int_0^Tg(\mu(t))\,\mathrm{d}t-C_1\right)
\end{equation}
and the Lagrange multiplier $\lambda\in\mathbb{R}$, $\delta L(\mu(t),\lambda)=0$ implies that
\begin{equation}
\mathcal{G}(\mu)=C_1,\,\Phi_h+\lambda g'(\mu(t)) = 0.
\end{equation}
If $\Phi_h\neq0$, these necessary conditions transform to \eqref{eq:adj-p1} and vice versa with $\lambda_{\mathrm{ref}} = {\lambda}/\Phi_h$ and identical $\mu$.
\end{proof}

The scalar $\Phi_h$ equals the Lie derivative of $\Phi$ in the time-rescaled system at the terminal time. For the five candidate $\mu$ listed in Table~\ref{FPs-P1} and with $\Phi(z(t))=|\hat{r}(t)|$, Fig.~\ref{fig:Finite_rescaled_centroid-finite} shows that $\Phi_h\ne 0$ only for the constant $\mu$. The remaining choices correspond to stationary values of $\Phi(z(T))$ given that $\mathcal{I}(\mu)=\tau^\ast$ and $\mathcal{G}(\mu)=C_1$ for $\tau^\ast$ corresponding to the local maximum and minimum, respectively, in the graph in Fig.~\ref{fig:Finite_rescaled_centroid-finite}. For such $\tau^\ast$, more than one $\mu$ thus gives rise to the same stationary value of the objective function, consistent with the observation made in conjunction with Table~\ref{FPs-P1}.





\begin{theorem}
\label{theo2}
Given $T\in\mathbb{R}_+$ and $\underline{\Phi}\in\mathbb{R}$, let $\mu^\ast\in C([0,T],\mathbb{R}_+)$ satisfy the first-order necessary conditions for a stationary point of the functional $\mu\mapsto\mathcal{G}(\mu)$ subject to  $\Phi(z(T))=\underline{\Phi}$. Then, if $\Phi_h\neq0$, it follows that $\mu^\ast$ satisfies the first-order necessary conditions for Problem~\ref{P2} and vice versa with $\Phi\left(\hat{z}\left(C_2\right)\right)-\underline{\Phi}=0$.
\end{theorem}
\begin{proof}
In terms of the Lagrangian
\begin{equation}
L := \int_0^Tg(\mu(t))\mathrm{d}t + \lambda\left(\Phi(z(T))-\underline{\Phi}\right)
\end{equation}
and the Lagrange multiplier $\lambda$, $\delta L(\mu(t),\lambda)=0$ implies that
\begin{equation}
\label{eq:theo2-adj-eqs}
\Phi(z(T))-\underline{\Phi}=0,\, \lambda\Phi_h+g'(\mu(t)) = 0.
\end{equation}
Provided that $\Phi_h\ne 0$, these necessary conditions transform to \eqref{eq:adj-p2} and vice versa with $C_2$ coupled to $\underline{\Phi}$ through
\begin{equation}
\label{eq:C2Phi}
\underline{\Phi}=\Phi(z(T))=\Phi(\hat{z}(\mathcal{I}(\mu)))=\Phi(\hat{z}(C_2)),
\end{equation}
$\lambda_{\mathrm{ref}}=\lambda\Phi_h$ and identical $\mu$.
\end{proof}


\begin{theorem}
\label{theo3}
Given $C_1\in\mathbb{R}_+$ and $\underline{\Phi}\in\mathbb{R}$, let $\mu^\ast$ and $T^\ast$ satisfy the first-order necessary conditions for a stationary point of the functional $(\mu,T)\mapsto T$ subject to the constraints $\mathcal{G}(\mu)=C_1$ and $\Phi(z(T))=\underline{\Phi}$. Then, if $\Phi_h\neq0$, it follows that $\mu^\ast$ and $T^\ast$ satisfy the first-order necessary conditions for Problem~\ref{P3} and vice versa with $\Phi\left(\hat{z}\left(C_2\right)\right)-\underline{\Phi}=0$.
\end{theorem}
\begin{proof}
In terms of the Lagrangian
\begin{align}
L &:= T+\lambda_1\left(\int_0^Tg(\mu(t))\,\mathrm{d}t-C_1\right) \nonumber\\
&\quad + \lambda_2\left(\Phi(z(T))-\underline{\Phi}\right)
\end{align}
and the Lagrange multipliers $\lambda_1,\lambda_2\in\mathbb{R}$, $\delta L(\mu(t),T,\lambda_1,\lambda_2)=0$ implies that
\begin{gather}
\mathcal{G}(\mu)-C_1=0,\, \Phi(z(T))-\underline{\Phi}=0,\label{theo3-adj-1-2}\\
 \lambda_1g'(\mu(t))+\lambda_2\Phi_h=0,\label{theo3-adj-3}\\
 1+\lambda_1g(\mu(T))+\lambda_2\Phi_h\mu(T)=0.\label{theo3-adj-4}
\end{gather}
Provided that $\Phi_h\ne 0$, these necessary conditions transform to \eqref{eq:adj-p3-l1}-\eqref{eq:adj-p3-T} and vice versa with $\lambda_{1,\mathrm{ref}}=\lambda_1$, $\lambda_{2,\mathrm{ref}}=\lambda_2\Phi_h$, identical $\mu(t)$, and $C_2$ coupled to $\underline{\Phi}$ through \eqref{eq:C2Phi}.
\end{proof}


\subsection{Application to the Kuramoto model}
\label{sec:ref2origin}

\subsubsection{Reduction of optimal control problems}
\label{sec:red-opt}
With $\Phi(z(T)):=|\hat{r}(T)|$ and $\underline{\Phi}=\underline{r}$, Theorems~\ref{theo1},~\ref{theo2} and~\ref{theo3} transform the search for candidate solutions to Problems~\ref{Maximizing consensus-Dynamic},~\ref{Minimize resource allocation-Dynamic} and~\ref{Minimum time control-Dynamic} to that of solutions to Problems~\ref{P1},~\ref{P2} and~\ref{P3}, provided that $\Phi_h\neq0$. Therefore, we have transferred the three optimization problems with differential constraints into reference problems without differential constraints, which are much easier to solve. Indeed, since the optimal $\mu$ for the reference optimization problems are constant (under certain non-degeneracy conditions on $g$), this must hold also for the three optimal control problems in Section~\ref{sec:optProbForm}, consistent with the preliminary observations in Section~\ref{sec:sim-finite-motivation}.

We can apply the solution to Problem~\ref{P1} directly to find the optimal $\mu$ for Problem~\ref{Maximizing consensus-Dynamic}. In contrast, for Problems~\ref{Minimize resource allocation-Dynamic} and~\ref{Minimum time control-Dynamic}, we must first find $C_2$ from the coupling condition \eqref{eq:C2Phi} for a given $\underline{r}$, for example using continuation in $C_2$.

\subsubsection{Validation of analytical solutions}
\label{sec:valid-finite}
Suppose again that $g(\mu)=\mu^2$. The optimal $\mu(t)\equiv\mu^\ast$ are given explicitly in \eqref{eq:sol2p4}, \eqref{eq:sol2p5}, and \eqref{eq:sol2p6}. Figures~\ref{fig:P1-T-C1},~\ref{fig:P2-T-r}, and~\ref{fig:P3-C1-r} show the predicted dependencies of $\mu^\ast$ on $T$ and $C_1$ for Problem~\ref{Maximizing consensus-Dynamic}, $T$ and $\underline{r}$ for Problem~\ref{Minimize resource allocation-Dynamic}, and $C_1$ and $\underline{r}$ for Problem~\ref{Minimum time control-Dynamic}, as well as results obtained using the successive continuation technique from Section~\ref{sec:sim-finite-motivation}. Note that the predicted results require the determination of $C_2$ for a given $\underline{r}$ from the implicit relationship \eqref{eq:C2Phi}. The agreement is excellent in all three cases.

\begin{figure}[ht]
\centering
\includegraphics[width=3.0in]{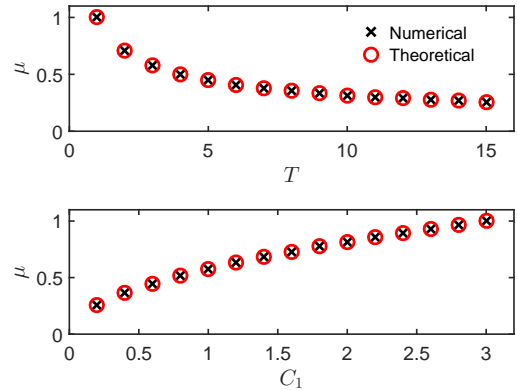}
\caption{The candidate solution $\mu^\ast=\sqrt{C_1/T}$ for Problem~\ref{Maximizing consensus-Dynamic} for varying $T$ (upper panel with $C_1=1$) and $C_1$ (lower panel with $T=3$). Here, and throughout the paper, crosses represent numerical data obtained using the successive continuation algorithm applied to the original optimization problem, while circles identify theoretical data predicted using the solution for the corresponding reference problem.}
\label{fig:P1-T-C1}
\end{figure}

\begin{figure}[ht]
\centering
\includegraphics[width=3.0in]{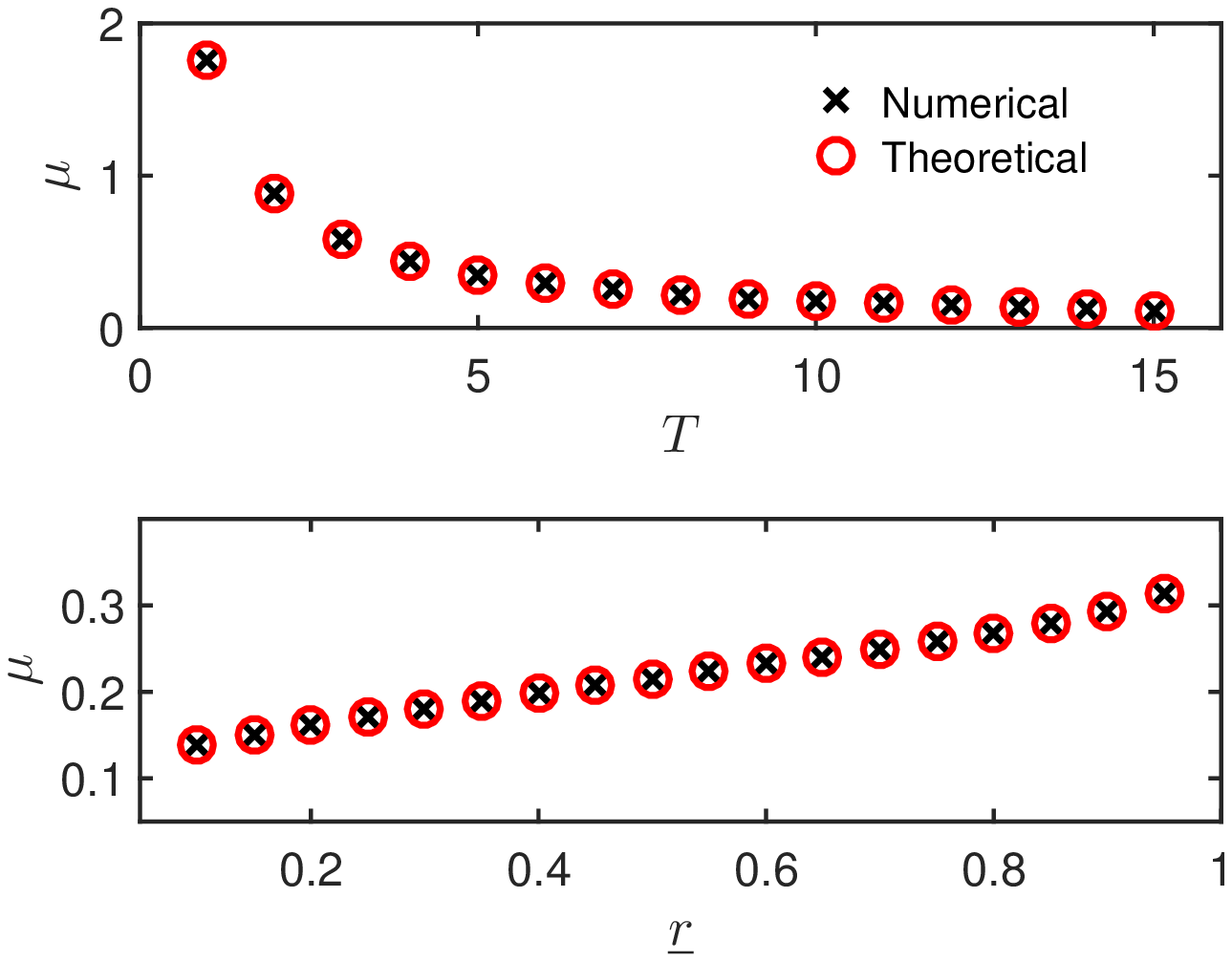}
\caption{The candidate solution $\mu^\ast=C_2(\underline{r})/T$ for Problem~\ref{Minimize resource allocation-Dynamic} for varying $T$ (upper panel with $\underline{r}=0.9$) and $\underline{r}$ (lower panel with $T=6$).}
\label{fig:P2-T-r}
\end{figure}

\begin{figure}[ht]
\centering
\includegraphics[width=3.0in]{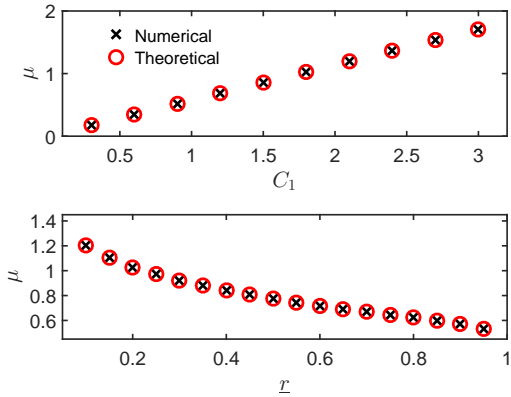}
\caption{The candidate solution $\mu^\ast=C_1/C_2(\underline{r})$ for Problem~\ref{Minimum time control-Dynamic} for varying $C_1$ (upper panel with $\underline{r}=0.9$) and $\underline{r}$ (lower panel with $C_1=1$).}
\label{fig:P3-C1-r}
\end{figure}

\subsubsection{Dependence on initial conditions and problem parameters}
\label{sec:sens-ic-p}
It follows from Theorem~\ref{theo1} that the optimal $\mu$ depends only on $T$ and $C_1$ and is thus independent of the initial conditions $z_0$ and problem parameters $p$. By extension, the optimal solution of Problem~\ref{Maximizing consensus-Dynamic} is also independent of the initial conditions $x_i(0)$. 

In contrast, it follows from Theorems~\ref{theo2} and~\ref{theo3} that the optimal $\mu$ depend on $T$, $\underline{\Phi}$, $z_0$,
and $p$ due to the coupling condition \eqref{eq:C2Phi}. By extension, the optimal solutions to Problems~\ref{Minimize resource allocation-Dynamic} and~\ref{Minimum time control-Dynamic} depend on the initial conditions $x_i(0)$.

\subsection{Optimality of stationary solutions}
We consider the possible optimality of $\Phi(z(T))$ for the candidate $\mu$ obtained from the solutions to Problems~\ref{P1},~\ref{P2}, and~\ref{P3} with $g(\mu)=\mu^2$.

For Problem~\ref{P1}, consider a series expansion of $\mu$ in terms of a complete orthonormal basis with coefficients $c_i$ and the first basis function equal to the constant $1/\sqrt{T}$. It follows that $\mathcal{I}(\mu)=c_1\sqrt{T}$ and $\mathcal{G}(\mu)=\sum_ic_i^2$. On the sphere $\mathcal{G}(\mu)=C_1$, $\tau(T)=\mathcal{I(\mu)}\leq \sqrt{C_1T}$ and $\mathcal{I}(\mu)$ is stationary for $\mu(t)\equiv\pm\sqrt{C_1/T}$, consistent with the solution of Problem~\ref{P1}. By Theorem~\ref{theo1}, it follows that the positive solution corresponds to a local maximum of $\Phi(z(T))$ if $\Phi_h>0$ and a local minimum if $\Phi_h<0$. From Fig.~\ref{fig:Finite_rescaled_centroid-finite}, we conclude that the constant $\mu$ listed in Table~\ref{FPs-P1} corresponds to a local maximum of $|\hat{r}(T)|$.

For Problem~\ref{P2}, note that the constraint $\Phi(z(T))=\underline{\Phi}$ is equivalent to $\mathcal{I}(\mu)=C_2$ in the time-rescaled system. It follows that the constraint manifold is an affine space in the coefficients $c_i$. By the convexity of $\mathcal{G}(\mu)$, the stationary solution $\mu(t)\equiv C_2/T$ corresponds to a global minimum of $\mathcal{G}(\mu)$~\cite{luenberger1997optimization}. By Theorem~\ref{theo2}, this $\mu$ is a global minimum also under the original constraint. It follows that the solution to Problem~\ref{Minimize resource allocation-Dynamic} found in Section~\ref{sec:sim-finite-motivation} is a global minimum, since here $C_2\approx1.7527$.

For Problem~\ref{P3}, we lack a general theory and leave this for future study.
The successive continuation approach shows that the stationary solution found for Problem~\ref{Minimum time control-Dynamic} is a local minimum under variations only in $p_1$. It follows that the corresponding $\mu$ does not correspond to a maximum of $T$, but we cannot \emph{a priori} exclude the possibility of a saddle.

\section{Additional applications}
\label{sec:two-more-appl}
\subsection{Synchronization in the continuum limit}
\subsubsection{System dynamics}
We set the natural frequency in the derivation in~\cite{2016kuramoto,scale-free-kuramoto} to zero and generalize $\mu$ to $\mu(t)$ in order to obtain a description of the Kuramoto model \eqref{eq:ode-degree} of an uncorrelated network of phase oscillators with degree distribution $p(k)$ in the continuum limit. This can be described in terms of the population density $\rho(x,t|k)$ of oscillators that have phase $x+\omega t$ at time $t$ for a \emph{given} degree $k$ and the \emph{continuity equation} 
\begin{equation}
\label{eq:contEqs}
\frac{\partial \rho}{\partial t}+\mu(t) k\frac{\partial }{\partial x}\left(\rho\Im(re^{-\mathrm{i}x})\right)=0,
\end{equation}
where the \emph{order parameter}
\begin{align}
\label{eq:orderPar}
r(t) &:= \frac{1}{\langle k\rangle}\int \mathrm{d}k'\int \mathrm{d}x'\, {k'p(k')}\rho(x',t|k')  e^{\mathrm{i}x'},
\end{align}
the mean $\langle k\rangle:=\int k'p(k')\,\mathrm{d}k'$, and $\int \rho(x,t|k)\,\mathrm{d}x =1$ for all $t$ and $k$.

Following Ott-Antonsen reduction~\cite{ott}, we restrict to
\begin{equation}
    \rho(x,t|k) = \frac{1+\sum_{n=1}^\infty [\alpha(t,k)]^ne^{\mathrm{i}n x}+c.c.}{2\pi},
\end{equation}
where $\alpha$ is a complex-valued function. This ansatz satisfies \eqref{eq:contEqs} provided that
\begin{equation}
\label{eq:pde-config1}
\dot{\alpha}+\frac{\mu(t) k}{2}(r\alpha^2-r^\ast)=0,
\end{equation}
where, now,
\begin{equation}
\label{eq:orderParRed}
    r(t)=\frac{1}{\langle k\rangle}\int\mathrm{d}k'\,k' p(k')\,\alpha^\ast(t,k')
\end{equation}
and $^\ast$ denotes complex conjugation. The system dynamics in the reduced manifold are thus governed by a \emph{partial integro-differential} equation.

To apply the established analytical framework, we focus on the case of $M$ distinct degree classes. To this end, let $\{k_i\}_{i=1}^M$ represent the set of corresponding degrees, and let $\hat{p}_i$ denote the fraction of nodes with degree $k_i$ in the \emph{whole} population, such that $\sum_{i=1}^M \hat{p}_i=1$. With $\alpha_i(t):=\alpha(t,k_i)$, the partial integro-differential equation then reduces to a set of ordinary differential equations
\begin{equation}
\label{eq:alpha-conf-t}
\dot{\alpha}_i+\frac{\mu(t)k_i}{2}(r\alpha^2_i-r^\ast)=0,\,\, i=1,...,M
\end{equation}
in terms of the order parameter 
\begin{equation}
    r = \frac{1}{\langle k\rangle}\sum_{j=1}^M {k_j\hat{p}_j} \alpha^\ast_j.
\end{equation}

\subsubsection{Characterization of synchronization}
In the continuum limit, the \emph{centroid} in the complex plane of the phase oscillators is given by
\begin{align}
\hat{r}(t) & :=\int \mathrm{d}k'\int \mathrm{d}x' {p(k')}\rho(x',t|,k')  e^{\mathrm{i}x'}
\end{align}
which reduces to $\hat{r}=\sum_{j=1}^M \hat{p}_j \alpha^\ast_j$ in the case of finitely many subpopulations. We use its amplitude $|\hat{r}|$ to characterize the synchronization level.

\subsubsection{Optimization of dynamics}
We again consider the three optimal control problems defined in Section~\ref{sec:optProbForm}, albeit with the new definition of $\hat{r}$ and the corresponding system dynamics~\eqref{eq:alpha-conf-t} instead of \eqref{eq:ode-degree}. Here, we take $\hat{p}_i\sim k_i^{-\gamma}$ motivated by the common use of power-law distributions to characterize empirical networks~\cite{powerLaw}. The exponent $\gamma$ acts as a problem parameter for the system dynamics.

As evident by inspection of the real and imaginary parts of~\eqref{eq:alpha-conf-t}, this non-autonomous dynamical system is separable. We conclude that the implications of Theorems~\ref{theo1}-\ref{theo3} still hold. In particular, with $\Phi(z(T)):=|\hat{r}(T)|$ and $\underline{\Phi}=\underline{r}$, the search for solutions to Problems~\ref{Maximizing consensus-Dynamic},~\ref{Minimize resource allocation-Dynamic} and~\ref{Minimum time control-Dynamic} is again replaced with the search for solutions to Problems~\ref{P1},~\ref{P2} and~\ref{P3}, provided that $\Phi_h\neq0$. We again expect that the optimal $\mu$ are constant functions that are independent of the initial conditions $\alpha_i(0)$ and the problem parameter $\gamma$ in the case of  Problem~\ref{Maximizing consensus-Dynamic}, and dependent on both $\alpha_i(0)$ and $\gamma$ in the case of Problems~\ref{Minimize resource allocation-Dynamic} and~\ref{Minimum time control-Dynamic}. 

These predictions are verified by applying the computational optimization technique used in Section~\ref{sec:sim-finite-motivation} to the corresponding optimal control problems. Consider, for illustration, the case of $M=10$, $k_i=i$, and initial conditions $\alpha_j(0)=\alpha_0 e^{\mathrm{i}2\pi(j-1)/M}$ for some positive $\alpha_0$. \textcolor{black}{It follows from numerical simulation of the time-rescaled dynamics that $\Phi(\hat{z}(\tau))$ is monotonically increasing with $\tau$, i.e., that $\Phi_h>0$.}


For Problem~\ref{Maximizing consensus-Dynamic}, the predicted optimal $\mu(t)\equiv\sqrt{C_1/T}$ is independent of the system dynamics (and, consequently, of the initial conditions and problem parameter). The results of the computational analysis are not graphed here, since they are consistent with Fig.~\ref{fig:P1-T-C1}. The independence with respect to $\alpha_0$ and $\gamma$ is shown in Fig.~\ref{fig:P1-inf-alpha-gamma-mu}.

\begin{figure}[ht]
\centering
\includegraphics[width=3.0in]{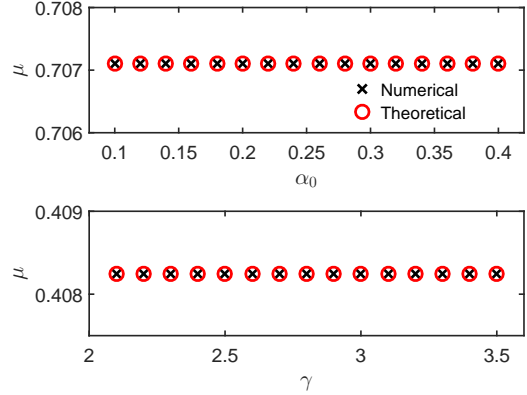}
\caption{The candidate solution $\mu^\ast=\sqrt{C_1/T}$ for Problem~\ref{Maximizing consensus-Dynamic} in the continuum limit for varying $\alpha_0$ (upper panel with $\gamma=2.2$, $T=6$, and $C_1=3$) and $\gamma$ (lower panel with $\alpha_0=0.1$, $T=6$, and $C_1=1$).}
\label{fig:P1-inf-alpha-gamma-mu}
\end{figure}

For Problem~\ref{Minimize resource allocation-Dynamic}, the predicted optimal $\mu(t)\equiv C_2/T$ requires the determination of $C_2$ in terms of $\underline{r}$ from the rescaled system dynamics and the coupling condition \eqref{eq:C2Phi}. We obtain similar agreement to that in Fig.~\ref{fig:P2-T-r}, albeit for a different functional dependence on $T$ and $\underline{r}$. The predicted dependence on $\alpha_0$ and $\gamma$ is validated by the results in Fig.~\ref{fig:P2-inf-alpha0-gamma-mu}.

\begin{figure}[ht]
\centering
\includegraphics[width=3.0in]{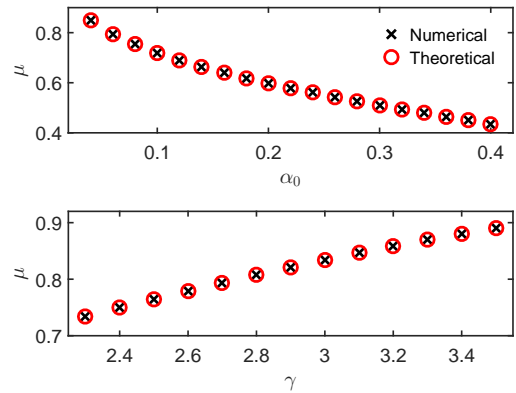}
\caption{The candidate solution $\mu^\ast=C_2(\alpha_0,\gamma,\underline{r})/T$ for Problem~\ref{Minimize resource allocation-Dynamic} in the continuum limit for varying $\alpha_0$ (upper panel with $\gamma=2.2$, $T=6$, and $\underline{r}=0.9$) and $\gamma$ (lower panel with $\alpha_0=0.1$, $T=6$, and $\underline{r}=0.9$).}
\label{fig:P2-inf-alpha0-gamma-mu}
\end{figure}

For Problem~\ref{Minimum time control-Dynamic}, the predicted optimal $\mu(t)\equiv C_1/C_2$ again requires the determination of $C_2$ in terms of $\underline{r}$ from the rescaled system dynamics and the coupling condition \eqref{eq:C2Phi}. We obtain similar agreement to that in Fig.~\ref{fig:P3-C1-r}, albeit for a different functional dependence on $C_1$ and $\underline{r}$. Fig.~\ref{fig:P3-inf-C1-r-T} validates the predicted dependence of the optimal time $T=C_2^2/C_1$ on $C_1$ and $\underline{r}$.

\begin{figure}[ht]
\centering
\includegraphics[width=3.0in]{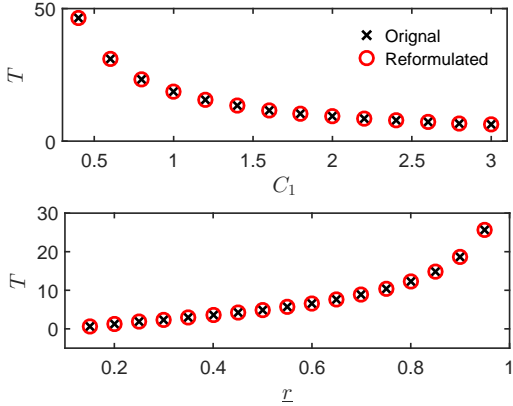}
\caption{The stationary value $T=C_2^2(\alpha_0,\gamma,\underline{r})/C_1$ for Problem~\ref{Minimum time control-Dynamic} in the continuum limit for varying $C_1$ (upper panel with $\alpha_0=0.1$, $\gamma=2.2$, and $\underline{r}=0.9$) and $\underline{r}$ (lower panel with $\alpha_0=0.1$, $C_1=1$, and $\gamma=2.2$).}
\label{fig:P3-inf-C1-r-T}
\end{figure}

\subsection{Spreading dynamics on activity-driven networks}

\subsubsection{Activity driven networks (ADNs)~\cite{activity-driven}}
Consider a network in the continuum limit, in which each node interacts with randomly selected other nodes at a constant rate $a$ per unit time, sampled from a probability distribution $p(a)$.

\subsubsection{Spreading dynamics}
Consider a susceptible-infected (SI) model of information spreading with $I(a,t)$ and $S(a,t)$ equal to the fractions of infected and susceptible agents with activity rate $a$ at time $t$. The dynamics is governed by the partial integro-differential equation (generalizing $\beta$ in \cite{activity-driven} to $\beta(t)$)
\begin{equation}
\label{eq:PDE-spreading}
 \frac{\partial I}{\partial t} =\beta(t)(1-I)(a\langle I\rangle+\langle aI\rangle)
\end{equation}
where $\langle\cdot\rangle$ denotes the expected value w.r.t.~$p(a)$ and $\beta$ is the transmission probability of information. Here, $a\langle I\rangle$ and $\langle aI\rangle$ correspond to \emph{active} and \emph{passive} mechanisms of infection, respectively.

To apply the established analytical framework, we focus on the case of $M$ distinct rate classes. To this end, let $\{a_i\}_{i=1}^M$ represent the set of corresponding interaction rates, and let $\hat{p}_i$ denote the fraction of nodes with rate $a_i$ in the \emph{whole} population, such that $\sum_{i=1}^M \hat{p}_i=1$. With $I_i(t):=I(a_i,t)$, the partial integro-differential equation then reduces to a set of ordinary differential equations~\cite{zino2016}
\begin{equation}
\dot{I}_i=\beta(t)(1-I_i) \left(a_i\langle I\rangle +\langle aI\rangle\right),\, i=1,\cdots,M.
\end{equation}

\subsubsection{Optimization of dynamics}

With $\beta(t)$ as the control input and for given initial conditions, we consider the following three optimal control problems (cf.~Section~\ref{sec:optProbForm}).

\begin{problem}[Maximum spread]\label{Maximizing spreading-ADN}
Given $T,C_1\in\mathbb{R}_+$, find
\begin{align}
\label{eq:p7}
 & \argmax_{\beta\in C([0,T],\mathbb{R}_+)} \langle I\rangle(T) \quad \text{s.t.}\quad \mathcal{G}(\beta)= C_1.
\end{align}
\end{problem}

\begin{problem}[Minimum effort]\label{Minimize resource allocation-ADN}
Given $T,\underline{I}\in\mathbb{R}_+$, find
\begin{align}
\label{eq:p8}
 & \argmin_{\beta\in C([0,T],\mathbb{R}_+)} \mathcal{G}(\beta) \quad\text{s.t.}\quad\langle I\rangle(T)=\underline{I}.
\end{align}
\end{problem}

\begin{problem}[Minimum time]\label{Minimum time control-ADN}
Given $C_1,\underline{I}\in\mathbb{R}_+$, find
\begin{align}
\label{eq:p9}
 & \argmin_{\beta\in C([0,T],\mathbb{R}_+)} T \quad\text{s.t.}\quad \mathcal{G}(\beta)=C_1,\,\langle I\rangle(T)=\underline{I}.
\end{align}
\end{problem}

With $\Phi(z(T)):=\langle I\rangle(T)$ and $\underline{\Phi}:=\underline{I}$, the analytical framework developed in Section~\ref{sec:analytical} again applies, since the governing dynamic system is separable. The observations from Section~\ref{sec:red-opt} and~\ref{sec:sens-ic-p} apply by replacing Problem~\ref{Maximizing consensus-Dynamic} with Problem~\ref{Maximizing spreading-ADN}, Problem~\ref{Minimize resource allocation-Dynamic} with Problem~\ref{Minimize resource allocation-ADN}, and Problem~\ref{Minimum time control-Dynamic} with Problem~\ref{Minimum time control-ADN}.

We perform validation for the analytical solution of optimal $\beta$ in a way similar to Section~\ref{sec:valid-finite}. The details of such validation are not presented here given their similarity to the ones in Section~\ref{sec:valid-finite}. Instead, we validate the analytical solutions for the Lagrange multipliers, which provide the sensitivity of the objective functional with respect to constraints. In the following analysis, we take \textcolor{black}{$\hat{p}_i\sim a_i^{-\gamma}$~\cite{activity-driven},} $\gamma=2.2$, $M=5$, $a_i=0.2+0.4(i-1)$ and initial condition $I_i(0)=0.02$. \textcolor{black}{It follows from numerical simulation of the time-rescaled dynamics that $\Phi(\hat{z}(\tau))$ is monotonically increasing with $\tau$, i.e., that $\Phi_h>0$.}

For Problem~\ref{Maximizing spreading-ADN}, it follows from the proof of Theorem~\ref{theo1} that the Lagrange multiplier $\lambda$ associated with the budget constraint $\mathcal{G}(\beta)=C_1$ is given by $\lambda=\lambda_{\mathrm{ref}}\Phi_h$, where $\lambda_{\mathrm{ref}} = -\sqrt{T/C_1}/2$. This analytical solution is validated in Fig.~\ref{fig:P7-T-C1-lambda}, where the dependence of $\lambda$ on $T$ and $C_1$ are presented.

\begin{figure}[ht]
\centering
\includegraphics[width=3.0in]{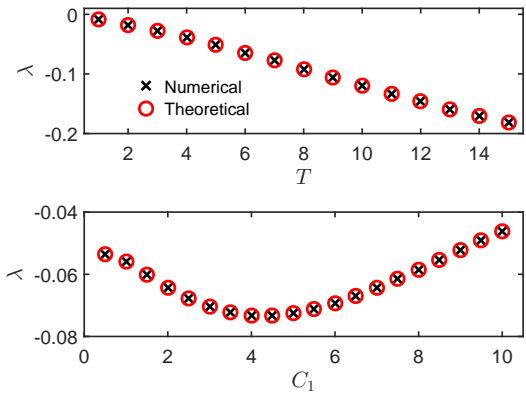}
\caption{The Lagrange multiplier $\lambda=-0.5\Phi_h\sqrt{T/C_1}$ to constraint $\mathcal{G}(\beta)=C_1$ in Problem~\ref{Maximizing spreading-ADN} for varying $T$ (upper panel with $C_1=2$) and $C_1$ (lower panel with $T=6$).}
\label{fig:P7-T-C1-lambda}
\end{figure}

For Problem~\ref{Minimize resource allocation-ADN}, it follows from the proof of Theorem~\ref{theo2} that the Lagrange multiplier $\lambda$ associated with the constraint $\langle I\rangle(T)=\underline{I}$ is given by $\lambda=\lambda_{\mathrm{ref}}/\Phi_h$, where $\lambda_{\mathrm{ref}} = - 2C_2/T$ and $C_2$ is related to $\underline{I}$ through the coupling condition \eqref{eq:C2Phi}. This analytical solution is validated in Fig.~\ref{fig:P8-T-I-lambda}, where the dependence of $\lambda$ on $T$ and $\underline{I}$ are presented.

\begin{figure}[ht]
\centering
\includegraphics[width=3.0in]{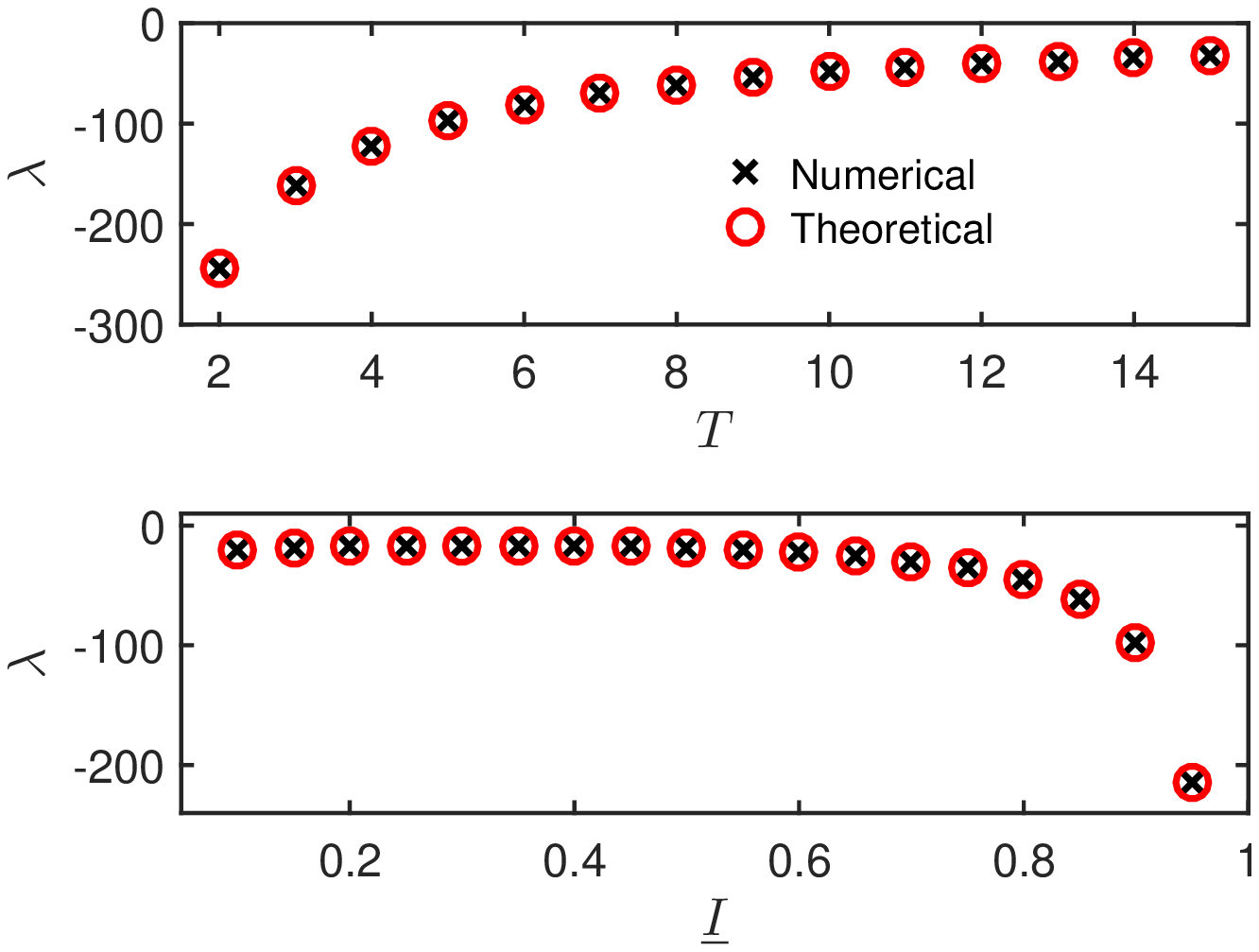}
\caption{The Lagrange multiplier $\lambda=-2C_2(\underline{I})/(T\Phi_h)$ to constraint $\langle I\rangle(T)=\underline{I}$ in Problem~\ref{Minimize resource allocation-ADN} for varying $T$ (upper panel with $\underline{I}=0.9$) and $\underline{I}$ (lower panel with $T=5$).}
\label{fig:P8-T-I-lambda}
\end{figure}

We denote the Lagrange multipliers to constraints $\mathcal{G}(\beta)=C_1$ and $\langle I\rangle(T)=\underline{I}$ in Problem~\ref{Minimum time control-ADN} to be $\lambda_1$ and $\lambda_2$ respectively.
It follows that $\lambda_1=\lambda_{1,\mathrm{ref}}$ and $\lambda_2=\lambda_{2,\mathrm{ref}}/\Phi_h$, as derived in the proof of Theorem~\ref{theo3}, where analytical $\lambda_{\mathrm{ref},1}$ and $\lambda_{\mathrm{ref},2}$ can be found in \eqref{eq:sol2p6-adj}. These analytical solutions are validated in Fig.~\ref{fig:P9-C1-I-lambda1} and Fig.~\ref{fig:P9-C1-I-lambda2}, where the dependence of $\lambda_1$ and $\lambda_2$ on some parameters are presented, respectively.

\begin{figure}[ht]
\centering
\includegraphics[width=3.0in]{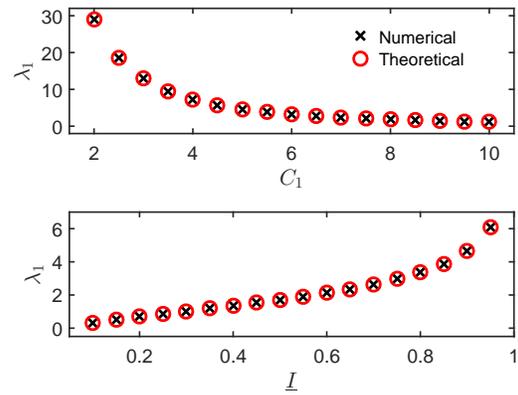}
\caption{Lagrange multiplier $\lambda_1=C_2^2(\underline{I})/C_1^2$ to the first constraint $\mathcal{G}(\beta)=C_1$ in Problem~\ref{Minimum time control-ADN} for varying $C_1$ (upper panel with $\underline{I}=0.9$) and $\underline{I}$ (lower panel with $C_1=5$).}
\label{fig:P9-C1-I-lambda1}
\end{figure}

\begin{figure}[ht]
\centering
\includegraphics[width=3.0in]{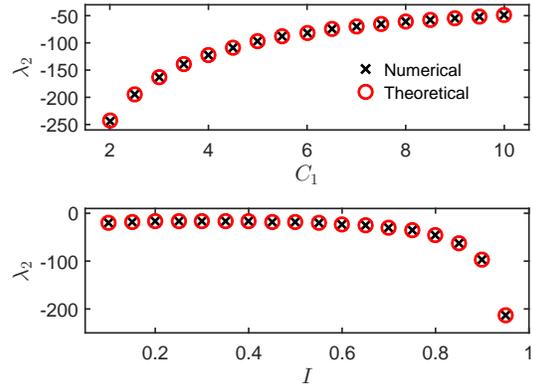}
\caption{Lagrange multiplier $\lambda_2=-2C_2(\underline{I})/(C_1\Phi_h)$ to the second constraint $\langle I\rangle(T)=\underline{I}$ in Problem~\ref{Minimum time control-ADN} for varying $C_1$ (upper panel with $\underline{I}=0.9$ ) and $\underline{I}$ (lower panel with $C_1=5$).}
\label{fig:P9-C1-I-lambda2}
\end{figure}

\section{Concluding Discussion}
\label{sec:conclusion}
We have established a unified analytical framework for a class of optimization problems whose dynamics are governed by separable non-autonomous systems. There are several opportunities for future study. Since we have assumed continuous control input in our framework, it is worth investigating the case of allowing discontinuities in the control input. One might further wish to consider optimization constrained also by inequality constraints, e.g., the boundedness of control input.

In Section~\ref{sec:two-more-appl}, we restricted attention to discrete distributions in order to apply the analytical framework established in Section~\ref{sec:analytical}. Nevertheless, this analytical framework can be extended also to the case of continuous distributions, where the dynamical system is governed by partial integro-differential equations, e.g., \eqref{eq:pde-config1} and \eqref{eq:PDE-spreading}. Specifically, consider a separable dynamical system $\partial{z}(t,a)/\partial t=\mu(t)h\left(z(t,a),\mathcal{F}(z(t,a),z(t,a'),a'),p\right)$, where $z:[0,\infty)\times\mathcal{J}_a\to\mathbb{R}^n$, $\mu:[0,\infty)\to\mathbb{R}_+$, $h:\mathbb{R}^n\times\mathbb{R}^n\times\mathbb{R}^s\mapsto\mathbb{R}^n$, and the nonlinear operator $
\mathcal{F}(z(t,a),z(t,a'),a'):=\int_{\mathcal{J}_a}  f\left(z(t,a),z(t,a'),a'\right)\mathrm{d}a'$ for some subset $\mathcal{J}_a$ of $\mathbb{R}$. Such a separable system can be reduced to an autonomous system by a straightforward generalization of Lemma~\ref{lema:odes}. It follows that Theorems~\ref{theo1},~\ref{theo2}, and~\ref{theo3} still apply by replacing $\Phi(z(T))$ with $\Phi(z(T,a),a):=\int_{\mathcal{J}_a} \phi(z(T,a),a)\,\mathrm{d}a$ and $\Phi_h$ with $\hat{\Phi}_{h}:=\int_{\mathcal{J}_a}\langle \phi_z(z(T,a),a),\underline{h}(T,a)\rangle\,\mathrm{d}a$, where
$\underline{h}(T,a)= h\left({z}(T,a),\mathcal{F}({z}(T,a),\hat{z}(T,a'),a'),p\right)$.



\section{Appendix}
\label{sec:appendix}

Let $g$ denote a positive, differentiable function such that $g'$ is not constant on any nonempty open interval.
\begin{lemma}\label{sec:sol2p1}
Consider the equations
\begin{equation}
\label{appeq:lemma1}
    \int_0^T g(\mu(t))\,\mathrm{d}t=C_1,\,1+\lambda g'(\mu(t))=0
\end{equation}
for $\mu\in C([0,T],\mathbb{R}_+)$. Then, if $C_1/T\in\range(g)$, $\mu$ equals some constant $\mu^\ast$. If $g'(\mu^\ast)\ne 0$, then $(\mu(t),\lambda)=(\mu^\ast,-1/g'(\mu^\ast))$ is a locally unique solution to \eqref{appeq:lemma1}.
\end{lemma}
\begin{proof}
Since $\lambda$ cannot equal $0$, it follows that $g'(\mu(t))$ must be constant. By the assumption on $g'$ and continuity of $\mu$ it follows that $\mu$ must equal some constant $\mu^\ast$. Substitution into the integral constraint yields $g(\mu^\ast)=C_1/T$, which can be locally uniquely inverted provided that $C_1/T\in\range(g)$ and $g'(\mu^\ast)\ne 0$. In this case,  $\lambda=-1/g'(\mu^\ast)$.
\end{proof}

\begin{lemma}\label{sec:sol2p2}
Consider the equations
\begin{equation}
\label{appeq:lemma2}
\int_0^T\mu(t)\,\mathrm{d}t=C_2,\, \lambda+g'(\mu(t)) = 0
\end{equation}
for $\mu\in C([0,T],\mathbb{R}_+)$. Then, $\mu$ equals some constant $\mu^\ast$ and $(\mu(t),\lambda)=(\mu^\ast,-g'(\mu^\ast))$ is a globally unique solution to \eqref{appeq:lemma2}.
\end{lemma}
\begin{proof}
By the assumption on $g'$ and continuity of $\mu$, the fact that $g'(\mu(t))=-\lambda$ is constant implies that $\mu$ must equal some constant $\mu^\ast$. Substitution into the integral constraint yields the unique solution $\mu^\ast=C_2/T$, from which it follows that $\lambda=-g'(\mu^\ast)$.
\end{proof}

\begin{lemma}\label{sec:sol2p3}
Consider the equations
\begin{gather}
\label{appeq:lemma3_1}
\int_0^Tg(\mu(t))\,\mathrm{d}t=C_1,\, \int_0^T\mu(t)\,\mathrm{d}t=C_2,\\ \lambda_1g'(\mu(t))+\lambda_2=0,\\
\label{appeq:lemma3_3}
\quad 1+\lambda_1g(\mu(T))+\lambda_2\mu(T)=0
\end{gather}
for $\mu\in C([0,T],\mathbb{R}_+)$. Then, if $C_1/C_2\in\range(\hat{g})$ for $\hat{g}(\mu):=g(\mu)/\mu$, $\mu$ equals some constant $\mu^\ast$. If $\gamma^\ast:=\mu^\ast g'(\mu^\ast)-g(\mu^\ast)\ne 0$, then
\begin{equation}
    (\mu(t),T,\lambda_1,\lambda_2)=\left(\mu^\ast,\frac{C_2}{\mu^\ast},\frac{1}{\gamma^\ast},-\frac{g'(\mu^\ast)}{\gamma^\ast}\right)
\end{equation}
is a locally unique solution to \eqref{appeq:lemma3_1}-\eqref{appeq:lemma3_3}.
\end{lemma}
\begin{proof}
Since $\lambda_1$ cannot equal $0$, it follows that $g'(\mu(t))$ must be constant. By the assumption on $g'$ and continuity of $\mu$ it follows that $\mu$ must equal some constant $\mu^\ast$. Substitution into the integral constraints yields $T^\ast=C_2/\mu^\ast$ and  $\hat{g}(\mu^\ast)=C_1/C_2$ , which can be locally uniquely inverted provided that $C_1/C_2\in\range(\hat{g})$ and $\hat{g}'(\mu^\ast)=\gamma^\ast/\mu^{\ast2}\ne 0$. In this case, $\lambda_1=1/\gamma^\ast$ and $\lambda_2=-g'(\mu^\ast)/\gamma^\ast$.
\end{proof}

\bibliographystyle{IEEEtran}
\bibliography{manuscript}

\begin{thebibliography}{10}
\providecommand{\url}[1]{#1}
\csname url@samestyle\endcsname
\providecommand{\newblock}{\relax}
\providecommand{\bibinfo}[2]{#2}
\providecommand{\BIBentrySTDinterwordspacing}{\spaceskip=0pt\relax}
\providecommand{\BIBentryALTinterwordstretchfactor}{4}
\providecommand{\BIBentryALTinterwordspacing}{\spaceskip=\fontdimen2\font plus
\BIBentryALTinterwordstretchfactor\fontdimen3\font minus
  \fontdimen4\font\relax}
\providecommand{\BIBforeignlanguage}[2]{{%
\expandafter\ifx\csname l@#1\endcsname\relax
\typeout{** WARNING: IEEEtran.bst: No hyphenation pattern has been}%
\typeout{** loaded for the language `#1'. Using the pattern for}%
\typeout{** the default language instead.}%
\else
\language=\csname l@#1\endcsname
\fi
#2}}
\providecommand{\BIBdecl}{\relax}
\BIBdecl

\bibitem{synchronization-auto}
F.~D{\"o}rfler and F.~Bullo, ``Synchronization in complex networks of phase
  oscillators: A survey,'' \emph{Automatica}, vol.~50, no.~6, pp. 1539--1564,
  2014.

\bibitem{2016kuramoto}
F.~A. Rodrigues, T.~K.~D. Peron, P.~Ji, and J.~Kurths, ``The {K}uramoto model
  in complex networks,'' \emph{Physics Reports}, vol. 610, pp. 1--98, 2016.

\bibitem{kuramoto1975}
Y.~Kuramoto, ``Self-entrainment of a population of coupled non-linear
  oscillators,'' in \emph{International symposium on mathematical problems in
  theoretical physics}.\hskip 1em plus 0.5em minus 0.4em\relax Springer, 1975,
  pp. 420--422.

\bibitem{synchronizationBook}
A.~Pikovsky, J.~Kurths, M.~Rosenblum, and J.~Kurths, \emph{Synchronization: a
  universal concept in nonlinear sciences}.\hskip 1em plus 0.5em minus
  0.4em\relax Cambridge university press, 2003, vol.~12.

\bibitem{strogatzLinearIncoherent}
S.~H. Strogatz and R.~E. Mirollo, ``Stability of incoherence in a population of
  coupled oscillators,'' \emph{Journal of Statistical Physics}, vol.~63, no.
  3-4, pp. 613--635, 1991.

\bibitem{crawford1994}
J.~D. Crawford, ``Amplitude expansions for instabilities in populations of
  globally-coupled oscillators,'' \emph{Journal of Statistical Physics},
  vol.~74, no. 5-6, pp. 1047--1084, 1994.

\bibitem{ott}
E.~Ott and T.~M. Antonsen, ``Low dimensional behavior of large systems of
  globally coupled oscillators,'' \emph{Chaos: An Interdisciplinary Journal of
  Nonlinear Science}, vol.~18, no.~3, p. 037113, 2008.

\bibitem{leander2015controlling}
R.~Leander, S.~Lenhart, and V.~Protopopescu, ``Controlling synchrony in a
  network of {K}uramoto oscillators with time-varying coupling,'' \emph{Physica
  D: Nonlinear Phenomena}, vol. 301, pp. 36--47, 2015.

\bibitem{slowdown}
M.~Karsai, M.~Kivel{\"a}, R.~K. Pan, K.~Kaski, J.~Kert{\'e}sz, A.-L.
  Barab{\'a}si, and J.~Saram{\"a}ki, ``Small but slow world: How network
  topology and burstiness slow down spreading,'' \emph{Physical Review E},
  vol.~83, no.~2, p. 025102, 2011.

\bibitem{lifetime_reference}
M.~Li, V.~D. Rao, T.~Gernat, and H.~Dankowicz, ``Lifetime-preserving reference
  models for characterizing spreading dynamics on temporal networks,''
  \emph{Scientific Reports}, vol.~8, no.~1, p. 709, 2018.

\bibitem{vertexburst}
L.~E. Rocha and V.~D. Blondel, ``Bursts of vertex activation and epidemics in
  evolving networks,'' \emph{PLoS Computational Biology}, vol.~9, no.~3, p.
  e1002974, 2013.

\bibitem{non-station}
D.~X. Horv{\'a}th and J.~Kert{\'e}sz, ``Spreading dynamics on networks: the
  role of burstiness, topology and non-stationarity,'' \emph{New Journal of
  Physics}, vol.~16, no.~7, p. 073037, 2014.

\bibitem{epidemic_control}
C.~Nowzari, V.~M. Preciado, and G.~J. Pappas, ``Analysis and control of
  epidemics: A survey of spreading processes on complex networks,'' \emph{IEEE
  Control Systems Magazine}, vol.~36, no.~1, pp. 26--46, 2016.

\bibitem{nowzari2017optimal}
------, ``Optimal resource allocation for control of networked epidemic
  models,'' \emph{IEEE Transactions on Control of Network Systems}, vol.~4,
  no.~2, pp. 159--169, 2017.

\bibitem{yang2016optimal}
L.-X. Yang, M.~Draief, and X.~Yang, ``The optimal dynamic immunization under a
  controlled heterogeneous node-based sirs model,'' \emph{Physica A:
  Statistical Mechanics and its Applications}, vol. 450, pp. 403--415, 2016.

\bibitem{activity-driven}
N.~Perra, B.~Gon{\c{c}}alves, R.~Pastor-Satorras, and A.~Vespignani, ``Activity
  driven modeling of time varying networks,'' \emph{Scientific Reports},
  vol.~2, p. 469, 2012.

\bibitem{zino2016}
L.~Zino, A.~Rizzo, and M.~Porfiri, ``Continuous-time discrete-distribution
  theory for activity-driven networks,'' \emph{Physical Review Letters}, vol.
  117, no.~22, p. 228302, 2016.

\bibitem{HongStrogatz}
H.~Hong and S.~H. Strogatz, ``Conformists and contrarians in a {K}uramoto model
  with identical natural frequencies,'' \emph{Physical Review E}, vol.~84,
  no.~4, p. 046202, 2011.

\bibitem{DKK91}
J.~Kern{\'e}vez and E.~Doedel, ``Optimization in bifurcation problems using a
  continuation method,'' in \emph{Bifurcation: Analysis, Algorithms,
  Applications}.\hskip 1em plus 0.5em minus 0.4em\relax Springer, 1987, pp.
  153--160.

\bibitem{staged_adjoint}
M.~Li and H.~Dankowicz, ``Staged construction of adjoints for constrained
  optimization of integro-differential boundary-value problems,'' \emph{SIAM
  Journal on Applied Dynamical Systems}, vol.~17, no.~2, pp. 1117--1151, 2018.

\bibitem{coco}
\BIBentryALTinterwordspacing
F.~Schilder and H.~Dankowicz, ``\textsc{coco}.'' [Online]. Available:
  \url{http://sourceforge.net/projects/cocotools}
\BIBentrySTDinterwordspacing

\bibitem{coco-recipes}
H.~Dankowicz and F.~Schilder, \emph{Recipes for continuation}.\hskip 1em plus
  0.5em minus 0.4em\relax SIAM, 2013.

\bibitem{inequality-ncp}
M.~Li and H.~Dankowicz, ``Optimization with equality and inequality constraints
  using parameter continuation,'' \emph{arXiv preprint arXiv:1909.01422}, 2019,
  to appear on Applied Mathematics and Computation.

\bibitem{luenberger1997optimization}
D.~G. Luenberger, \emph{Optimization by vector space methods}.\hskip 1em plus
  0.5em minus 0.4em\relax John Wiley \& Sons, 1997.

\bibitem{scale-free-kuramoto}
S.~Yoon, M.~S. Sindaci, A.~Goltsev, and J.~Mendes, ``Critical behavior of the
  relaxation rate, the susceptibility, and a pair correlation function in the
  {K}uramoto model on scale-free networks,'' \emph{Physical Review E}, vol.~91,
  no.~3, p. 032814, 2015.

\bibitem{powerLaw}
A.~Clauset, C.~R. Shalizi, and M.~E. Newman, ``Power-law distributions in
  empirical data,'' \emph{SIAM Review}, vol.~51, no.~4, pp. 661--703, 2009.

\end{thebibliography}


\end{document}